\documentclass[11pt]{amsart}

\usepackage{amsfonts,amsmath,amssymb,amsthm}

\usepackage[all]{xy}
\xyoption{matrix}
\usepackage[svgnames,hyperref]{xcolor}
\usepackage[normalem]{ulem}
\usepackage{mathabx} 
\newcommand{\nicecolor}{Navy}  
\usepackage[shortlabels]{enumitem}
\setlist[1]{wide}
\setlist[2]{leftmargin=15mm} 
\setlist[enumerate]{label=\rm{(\arabic*)}}
\setlist[enumerate,2]{label=\rm({\it\roman*}), }
\setlist[itemize]{label=\raisebox{0.25ex}{\tiny$\bullet$}}
\usepackage[backref, colorlinks, linktocpage, citecolor = \nicecolor, linkcolor = \nicecolor, urlcolor = \nicecolor]{hyperref}

\usepackage{mathtools}

\newtheorem{mytheorem}{Theorem}

\newtheorem{lemma}{Lemma}[section]
\newtheorem{corollary}[lemma]{Corollary}
\newtheorem{theorem}[lemma]{Theorem}

\newtheorem{question}[lemma]{Question}

\theoremstyle{definition}
\newtheorem{definition}[lemma]{Definition}

\theoremstyle{remark}
\newtheorem{remark}[lemma]{Remark}

\newtheorem{example}[lemma]{Example}

\usepackage{extarrows}
\usepackage{mathtools}
\usepackage{graphicx}
\usepackage{tikz-cd}
\usepackage{tikz}

\newcommand\dashdownarrowi{\mathchoice%
    {\rotatebox[origin=c]{-90}{$\displaystyle\dashrightarrow$}}%
    {\rotatebox[origin=c]{-90}{$\displaystyle\dashrightarrow$}}%
    {\rotatebox[origin=c]{-90}{$\scriptstyle\dashrightarrow$}}%
    {\rotatebox[origin=c]{-90}{$\scriptscriptstyle\dashrightarrow$}}%
   } 
\newcommand{\dashdownarrow}{\mathrel{\dashdownarrowi}} 

\newcommand{\hookuparrow}{\mathrel{\rotatebox[origin=c]{90}{$\hookrightarrow$}}}

\author{Ahmed Abouelsaad}
\curraddr{Department of Mathematics and Computer Science, University of Basel, 4051 Basel, Switzerland}
\email{a.abouelsaad@unibas.ch}

\address{Mathematics Department, Faculty of Science, Mansoura University,
Mansoura 35516, Egypt}
\email{abuelsaad@mans.edu.eg}
\thanks{The author acknowledges support from the Federal Commission for Scholarships for Foreign Students (FCS)}

\begin{document}

\title{Galois Points and Cremona Transformations}

\maketitle 
\begin{abstract}

In this article, we study Galois points of plane curves and the extension of the corresponding Galois group to $\mathrm{Bir}(\mathbb{P}^2)$.\\
If the Galois group has order at most $3$, we prove that it always extends to a subgroup of the Jonquières group associated to the point $P$.\\
In degree at least $4$, we prove that it is false. We provide an example of a Galois extension whose Galois group is extendable to Cremona transformations but not to a group of de Jonquières maps with respect to $P$. We also give an example of a Galois extension whose Galois group cannot be extended to Cremona transformations.
\end{abstract}
\section{Introduction}%
Let $k$ be an algebraically closed field. Let $C$ be an irreducible plane curve in $\mathbb{P}^2$. Giving a point $P\in\mathbb{P}^2$, we consider the projection $\pi_P |_C: C\dashrightarrow\mathbb{P}^1$, which is the restriction of the projection $\pi_P:\mathbb{P}^2 \dashrightarrow \mathbb{P}^1$ with center $P$. Let $K_P = \pi^*_P(k(\mathbb{P}^1))$. We say that $P$ is Galois if $k(C)/K_{P}$ is Galois and denote by $G_{P}$ the corresponding Galois group in this case.\\[1mm]
A de Jonquières map is a birational map $\varphi$ for which there exist points $P, Q\in \mathbb{P}^2$ such that $\varphi$ sends the pencil of lines passing through $P$ to the pencil of lines passing through $Q$. The group of a de Jonquières transformations preserving the pencil of lines passing through a given point $P \in \mathbb{P}^2$ is denoted by $\mathrm{Jonq}_{P} \subset \mathrm{Bir}(\mathbb{P}^2)$.\\
As in \cite{Mr1048550}, we are interested in the extension of elements of $G_P$ to $\mathrm{Bir}(\mathbb{P}^2)$. There are two interesting questions:
\begin{question}\label{Q1}%
If $P$ is Galois, does $G_P$ extends to $\mathrm{Bir}(\mathbb{P}^2)$?
\end{question}
\begin{question}\label{Q2}
   If an element extends to $\mathrm{Bir}(\mathbb{P}^2)$, does it extend to a de Jonquières map? $i.e.$ to an element $\varphi\in \mathrm{Bir}(\mathbb{P}^2)$ with $\pi_{P}\circ\varphi=\pi_P$?
\end{question}
Consider a point $P\in \mathbb{P}^2$ with multiplicity $m_{P}$ on an irreducible plane curve $C$ in $\mathbb{P}^2$ of degree $d$, we will show later that the extension $[k(C):K_{P}]$ has degree $d-m_{P}$. Depending on the degree of the extension, we divided the answers to the preceding two questions. Consequently, our first main result is the following theorem.
\begin{mytheorem}\label{A1}
Let  $P\in \mathbb{P}^2$, let $C\subset \mathbb{P}^2$ be an irreducible curve. If the extension $k(C)/K_{P}$ is Galois of degree  at most $3$, then $G_{P}$ always extends to a subgroup of $ \mathrm{Jonq}_{P}\subseteq \mathrm{Bir}(\mathbb{P}^2)$.
\end{mytheorem} 
Theorem \ref{A1} resulted from Theorem \ref{thm:cases}, which provides more information on the Galois extensions of degree at most $3$ and the related Galois Groups at a point $P$. This encourages us to study higher-degree Galois extensions and determine if their Galois groups $G_P$ can always be extended to $\mathrm{Bir}(\mathbb{P}^2)$ as well as to the group of de Jonquières map with respect to $P$. The following theorem is our second argument. 
\begin{mytheorem}\label{A2}
    Let $k$ be a field of characteristic $char(k)\neq2$ containing a primitive fourth root of unity, and let $C$ be the irreducible curve defined by the equation 
\begin{equation}
X^{4}-4 Z Y X^{2}-Z Y^{3}+2 Z^{2} Y^{2}-Y Z^{3} =0,
\end{equation}
then the point $P=[1:0:0]$ is an outer Galois point of $C$ and the extension induced by the projection $\pi_P: C\dashrightarrow \mathbb{P}^1$  is Galois of degree $4$. The group $G_{P}$ extends to $\mathrm{Bir}(\mathbb{P}^2)$ but not to $\mathrm{Jonq}_{P}$.
\end{mytheorem}  
Theorem \ref{A2} (as a result of Lemma \ref{toBir}) provides a negative answer to Question \ref{Q2} for the Galois extension of degree at least $4$. Although the Galois Group $G_P$ can be extended to $\mathrm{Bir}(\mathbb{P}^2)$, this does not necessarily mean that it can be extended to $\mathrm{Jonq}_P$. We also get the result listed below.
\begin{mytheorem}\label{A3}
Let $k$ be a field with $char(k)\neq5$ that contains a primitive fifth root of unity, and let $\phi:\mathbb{P}^1 \rightarrow \mathbb{P}^2$ given by \[\phi:[u:v]\mapsto[u v^6-u^7:u^5(u^2+v^2): v^5(u^2+v^2)].\] We define $C:=\overline{\phi(\mathbb{P}^1)}$ which is an irreducible curve of $\mathbb{P}^2$, then the point $P=[1:0:0]$ is an inner Galois point of $C$ and the extension induced by the projection $\pi_P: C\dashrightarrow \mathbb{P}^1$  is Galois of degree $5$. Moreover, the identity is the only element of the Galois group that extends to $\mathrm{Bir}(\mathbb{P}^2)$.
\end{mytheorem}
Due to theorem \ref{A3}, we cannot claim that $G_P$ can always be extended to $\mathrm{Bir}(\mathbb{P}^2)$ for every Galois extension of degree at least $4$, see Lemma \ref{not-to-Bir}.
\section*{Acknowledgement}
I would like to thank my PhD advisor Jérémy Blanc for suggesting the question and for interesting discussions during the preparation of this text. I would also like to thank the Department of Mathematics and Computer Science at Basel for the hospitality.
\section{Preliminaries}
The concept of Galois points for irreducible plane curve $C\subset\mathbb{P}^2$ was introduced by \cite{miura2000field}, \cite{Yoshihara2009rational}, \cite{Fukasawa2009galois}. In order to study the extension of an element in $G_P$ to $\mathrm{Bir}(\mathbb{P}^2)$, we need the following definition.
\begin{definition}
Let $K_P = \pi^*_P(k(\mathbb{P}^1))$, the point P is called a Galois point for $C$ if the field extension $k(C)/K_P$ is Galois. Moreover, if $P\in C$ $[$resp. $P\notin C]$, then we call $P$ an inner [resp. outer] Galois point. If $P$ is Galois, we write $G_{P}=Gal(k(C)/K_P)$ and call it the Galois group at $P$.
\end{definition}
\begin{lemma}\label{lemm:deg}
 The field extension induced by $\pi_P$, $i.e$., $\pi^*_P: k(\mathbb{P}^1) \hookrightarrow k(C)$, $i.e.$ the extension  $k(C) /K_P$ has degree $d-m_P$, where $m_P$ is the multiplicity of $C$ at $P$, and $d$ is the degree of $C$.
\end{lemma}
\begin{proof}
Let $F(X, Y, Z) = 0$ be the defining equation of $C$ of degree $d$. By changing coordinates, we may fix the point $P$ to be the point $[1:0:0]\in \mathbb{P}^2$ and choose that $C$ is not the line $Z=0$. Since $P$ has multiplicity $m_P$, so we can write the equation of the curve in the following form.
\begin{equation}
    F(X,Y,Z)=F_{m_p}(Y,Z)X^{d-m_p}+......+F_{d-1}(Y,Z)X+F_{d}(Y,Z),
\end{equation}
where $F_i(Y,Z)$ is a homogeneous polynomial of $Y$ and $Z$ of degree $i~(m_P\leq i\leq d)$ and $F_{m_P}(Y,Z)\neq0$. Since $F(X,Y,Z)$ is irreducible in $k[X,Y,Z]$ and not a multiple of $Z$, then $f=F(X,Y,1)\in k[X,Y]$ is also irreducible in $k[X,Y]$. We take affine chart $Z=1$, so we can see $f$ as an irreducible polynomial in $\tilde{k}[X]$ with $deg_{X}(f)=d-m_P$ where $\Tilde{k}=k(Y)$. And hence the extension $k(C)/K_{P}$ is isomorphic to $(k(Y)[X]/(f))/(k(Y))$ and thus it has the same degree equal to the degree of the irreducible polynomial $f\in \tilde{k}[X]$, so we get $[k(C):\tilde{k}]=deg_{X}(f)=d-m_P$.
\end{proof} 
\begin{definition}
Let $E/F$ be a field extension. An element $\alpha\in E$ is said to be algebraic over $F$ if there is a non-constant polynomial $f(X) \in F[X]$ such that $f(\alpha)=0$; if $\alpha$ is not algebraic over $F$, it is said to be transcendental over $F$, and if every element of $E$ is algebraic over $F$, then $E$ is said to be an algebraic extension of $F$.
\end{definition}
\begin{definition}
Let $E/F$ be a field extension, and $\alpha\in E$ be an algebraic element. The monic polynomial of the least degree among all polynomials in the set $\{f(x)\in F[x];~f(\alpha)=0\}$ is the minimal polynomial of $\alpha$.
\end{definition}
\begin{definition}
Let $E/F$ be an algebraic extension, an element $\alpha\in E$ is said to be separable if its minimal polynomial $f(X) \in F[X]$ does not have multiple roots in an algebraic closure of $E$.  If every element of $E$ is separable, then $E$ is said to be separable.
\end{definition}
\begin{definition}
A rational map of $\mathbb{P}^n_{k}$ is a map of the type  $[z_0: z_1: ...: z_n]\dashrightarrow [\phi_{0}(z_0, z_1,...,z_n):\phi_{1}(z_0, z_1,...,z_n):....:\phi_{n}(z_0, z_1,...,z_n)]$ where the $\phi_i’s\in k[z_1,..,z_n]$ denote homogeneous polynomials of the same positive degree, we may always assume these to be without common factor.
\end{definition}
\begin{definition}
A birational map $\phi$ of $\mathbb{P}^n_{\mathbb{C}}$ is a rational map of $\mathbb{P}^n_{\mathbb{C}}$ such that there exists a rational map $\psi$ of $\mathbb{P}^n_{\mathbb{C}}$ with the following property $\phi\circ\psi=\psi\circ\phi=id,$ where $id: [z_0 : z_1 : ... : z_n]\stackrel{}{\dashrightarrow}[z_0 : z_1 :...: z_n]$. The degree of $\phi\in \mathrm{Bir}(\mathbb{P}^n_{\mathbb{C}})$ is the degree of the $\phi_i$s, when taken without common factor.
\end{definition}
\begin{theorem}\cite[Theorem 4.4]{MR0463157}
For any two varieties $X$ and $Y$, there is a bijection between the set of dominant rational maps $\varphi:X \dashrightarrow Y$, and the set of field homomorphisms $\varphi^{\star}:k(Y) \rightarrow k(X)$. This bijection is given by $\varphi\mapsto\varphi^{\star}$.
\[\begin{tikzcd}[column sep=huge,row sep=huge]
X\arrow[r,dashed,"\varphi",->] \arrow[dr,swap,"\varphi^{\star}(a)=a\circ \varphi",dashed,->] &
  Y \arrow[d,"a",dashed,->] \\
& k 
\end{tikzcd}\]\end{theorem} 
\begin{corollary}\label{cor_Bir}For each variety $X$, we have a group isomorphism $\mathrm{Bir}(X) \stackrel{\simeq }{\rightarrow}\mathrm{Aut}_{k}(k(X))$ which sends $\varphi$ to $\varphi^{*}$.
\end{corollary}
\begin {lemma}\label{LemBir1}
For any field $k$, we have $\mathrm{Aut}_k(k(x))=\{x\mapsto(a x+b)/(c x +d);~a,b,c,d\in k, ~a d-b c\neq0\}$ and $\mathrm{Aut}_k(k(x))\cong \mathrm{Bir}(\mathbb{A}^1)\cong \mathrm{Bir}(\mathbb{P}^1)$.
\end{lemma}
\begin{proof} As $k(x)$ is the function field of $\mathbb{A}^1$ and $\mathbb{P}^1$, we have $\mathrm{Bir}(\mathbb{A}^1)\cong \mathrm{Bir}(\mathbb{P}^1) \cong \mathrm{Aut}_k(k(x))$ by Corollary \ref{cor_Bir}. Let $a,b,c,d\in k$ and $a d-b c\neq0$, then the map $f$ given by $x\mapsto(a x+b)/(c x +d)$ has $g$ as an inverse given by $x\mapsto(d x-b)/(-c x +a)$ since $f\circ g=g\circ f=id$, hence $f\in \mathrm{Aut}_k(k(x))$.\\
Conversely, suppose $f\in \mathrm{Bir}(\mathbb{P}^1)$, given by $[u:v]\mapsto[f_1(u,v):f_2(u,v)]$, where $f_1$ and $f_2$ have the same degree $d_1$ and no common factor. If $f$ has an inverse $g$ of degree $d_2$, then the composition $f \circ g$ is of degree $d_1d_2$, as there is no simplification; otherwise a common factor of $f_{1}(g_{1},g_{2})$ and $f_{2}(g_{1},g_{2})$ would give a point of $\mathbb{P}^1$ sent by $g$ onto $(0,0)$ or onto a point of $\mathbb{P}^1$ sent by $f$ onto $[0:0]$. But  $f\circ g=id$, so $1=deg(id)= deg(f \circ g)=d_1d_2$. Since there is no simplification in $\mathbb{P}^1$, so $d_1 , d_2<2$, and hence $deg(f)=1$. Hence $f$ is given by $[u:v]\mapsto[au+bv:cu+dv]$ where $ad-bc\neq 0$.
\end{proof}
\begin {lemma} \label{Lemma:aut}
For any field $k$, let $f:\mathbb{A}^n\dashrightarrow\mathbb{A}^n$ be a birational map, given by $(x_{1},..,x_{n})\mapsto (f_1,..,f_n)$, where $f_i\in k(x_1,..,x_n), i=1,..,n$, then $k(x_1,..,x_n)=k(f_1,..,f_n)$.
\end {lemma}
\begin{proof}
Let $f$ be a birational map. By Corollary \ref{cor_Bir}, $f^{*}$ is an automorphism of the $k-$algebra $k(x_1,..,x_n)$, hence it is surjective, which yields $ k(x_1,..,x_n)=k(f_1,..,f_n)$.
\end{proof}
\begin {definition}
The group of birational maps of $\mathbb{P}^n_{\mathbb{C}}$ is denoted $\mathrm{Bir}(\mathbb{P}^n_{\mathbb{C}})$ and called the Cremona group. If $C \subset\mathbb{P}^2$ is an irreducible curve and $\phi\in \mathrm{Bir}(\mathbb{P}^2)$, we say that $\phi$ preserves $C$ (or leaves $C$ invariant) if $\phi$ restricts to a birational transformation of $C$. If this transformation is the identity, we say that $\phi$ fixes $C$.
\end{definition}
\begin {definition}Let $P\in\mathbb{P}^2$, we write \[\mathrm{Jonq}_{P}=\{\varphi\in \mathrm{Bir}(\mathbb{P}^2)|\exists~\alpha\in~\mathrm{Aut}(\mathbb{P}^1);\pi_P\circ\varphi=\alpha\circ\pi_P\}\] and call it the Jonquières group of $P$. A de Jonquières map with respect to $P$ is a birational map of $\mathbb{P}^2$ that belongs to $\mathrm{Jonq}_{P}$, this corresponding to ask that $\phi$ preserves a pencil of lines through the point $P$.
\end{definition}
\begin {lemma}
Let $P=[1:0:0]$, by taking an affine chart, a de Jonquières map with respect to $P$ is a special case of a Cremona transformation, of the form
\[\iota^{-1}\circ \mathrm{Jonq}_P\circ\iota=\{(x,y)\mapsto (\frac{a x+b}{c x +d}, \frac{r_1(x) y+ r_2(x)}{r_3(x) y+ r_4(x)})\}\]
and $ \iota:\mathbb{A}^2 \hookrightarrow \mathbb{P}^2,~ (x,y)  \longmapsto [x:y:1] $, where $a,b,c,d\in k$ with $a d-b c\neq0$  and $r_1(x),r_4(x),r_2(x), r_3(x)\in k(x)$ with $r_1(x) r_4(x)-r_2(x) r_3(x)\neq0$.
\end{lemma} 
\begin{proof}
We have the following commutative diagram
\begin{equation*}
   \begin{array}{ccc}
     \hspace{5mm} \mathbb{A}^2 &\stackrel{\iota }{\hookrightarrow}&\mathbb{P}^2\\
     (x,y)  &\longmapsto &[x:y:1]\\
   \pi_x \dashdownarrow& &  \dashdownarrow \pi_P\\
 \hspace{5mm} \mathbb{A}^1 &\stackrel{\psi }{\hookrightarrow}&\mathbb{P}^1\\
     x  &\longmapsto &[x:1]
    \end{array}
\end{equation*}
which gives the equality,
\[\iota^{-1}\circ \mathrm{Jonq}_P\circ\iota=\{f\in \mathrm{Bir}(\mathbb{A}^2)|\exists \alpha\in \mathrm{Bir}(\mathbb{A}^1);\alpha\circ\pi_x=\pi_x\circ f\}.\]
Let $f\in \iota^{-1}\circ \mathrm{Jonq}_P\circ\iota$ given by $(x,y)\mapsto (f_1(x,y),f_2(x,y))$, then $\pi_x\circ f:(x,y)\mapsto f_1(x,y)$, but $\alpha\circ\pi_x=\alpha(x)$, it follows that $f_1(x,y)$ depends only on $x$, and is of this form $f_1(x,y)=(a x+b)/(c x +d)$ where $a,b,c,d\in \mathbb{C}$ and $a d-b c\neq0$ by Lemma \ref{LemBir1}. From Lemma \ref{Lemma:aut}, we have $k(x,y)= k((a x+b)/(c x +d),f_2(x,y))$, to describe the second component $f_2(x,y)$, let us define the birational map $\tau:(x,y)\mapsto ((d x-b)/(-c x +a),y)$, hence $\tau\circ f:(x,y)\mapsto(x,f_2(x,y))$ is a birational map since both $f$ and $\tau$ are birational, by Lemma \ref{Lemma:aut} again we have $k(x)(f_2(x,y))=k(x)(y)$. Apply Lemma \ref{LemBir1} over the field $k(x)$ we get $f_2(x,y)=(r_1(x) y+ r_2(x))/(r_3(x) y+ r_4(x))$.
\end{proof}
\begin{lemma}\label{Lem:Galois1}
Let $P,Q\in \mathbb{P}^2$ and $C,D\subset \mathbb{P}^2$ be two irreducible curves, if $\phi \in \mathrm{Bir}(\mathbb{P}^2)$ and $\phi|_{C}:C\dashrightarrow D$ is birational map, we assume that there exists $\theta\in \mathrm{Aut}(\mathbb{P}^1)$ such that $\pi_{Q}\circ \phi=\theta\circ\pi_P$, then $P$ is a Galois point of $C$ if and only if $Q$ is a Galois point of $D$. Moreover, if $P$ is Galois, an element of $G_{P}$ extends an element of of $\mathrm{Bir}(\mathbb{P}^2)$ (respectively $\mathrm{Jonq}_{P}$) if and only if its image in $G_{Q}$ extends an element of of $\mathrm{Bir}(\mathbb{P}^2)$ (respectively $\mathrm{Jonq}_{Q}$)
\begin{equation*}
   \begin{array}{ccc}
     \hspace{5mm} C &\stackrel{\phi }{\dashrightarrow}&D \\
   \pi_P \dashdownarrow& &  \dashdownarrow \pi_Q\\
    \hspace{5mm} \mathbb{P}^1 & \stackrel{\theta }{\longrightarrow}&\mathbb{P}^1
    \end{array}
\end{equation*}
\end{lemma}
\begin{proof}
Since $\phi|_C$ is birational map from $C$ to $D$, then $\phi^{*}|_C :k(D)\rightarrow k(C)$ is an isomorphism. Moreover, as $\pi_Q\circ\phi=\theta\circ\pi_P$, we have a commutative diagram 
\begin{equation*}
   \begin{array}{ccc}
     \hspace{5mm}  k(D)&\stackrel{\phi^{*}|_C} {\longrightarrow}& k(C) \\[.8mm]
   \pi^{*}_Q \hookuparrow& &  \hookuparrow \pi^{*}_P\\[.8mm]
    \hspace{5mm} k(\mathbb{P}^1) & \stackrel{\theta^{*} }{\longrightarrow}&k(\mathbb{P}^1)
    \end{array}
\end{equation*}
Hence, $k(D)/\pi_{Q}^{*}(k(\mathbb{P}^1))$ is Galois if and only if $k(D)/\pi_{P}^{*}(k(\mathbb{P}^1))$ is Galois. Moreover, $\phi$ conjugates $\mathrm{Jonq}_{P}$ to $\mathrm{Jonq}_{Q}$ and sends any element of $\mathrm{Bir}(\mathbb{P}^2)$ that preserves $C$ onto element of $\mathrm{Bir}(\mathbb{P}^2)$ that preserves $D$.
\end{proof}
\begin{example}
Let $P=Q\in\{[1:0:0],[0:1:0],[0:0:1]\}~\phi:[X:Y:Z] \mapsto [Y Z:X Z:X Y] $ and $\theta:[Y:Z] \mapsto [ Z:Y] $, let $C\subset\mathbb{P}^2$ be an irreducible curve not equal to $x=0,~y=0$ or $z=0$ and let $D=\phi(C)$, so we have the following diagram 
 \begin{equation*}
   \begin{array}{ccc}
     \hspace{5mm} C &\stackrel{\phi }{\dashrightarrow}& D \\
     \pi_P \dashdownarrow& &  \dashdownarrow \pi_P\\
    \hspace{5mm} \mathbb{P}^1 & \stackrel{\theta }{\longrightarrow}&\mathbb{P}^1
    \end{array}
\end{equation*}
which shows that, if $P$ is a Galois point for $C$, then $P$  becomes a Galois point for $D$, this is a particular case of Lemma \ref{Lem:Galois1} corresponding to \cite[Corollary $3$]{MR2427627}.
\end{example}
\section{Extensions of degree at most three }
\begin{lemma}\label{Lemma:Galois}
Let $k$ be a field and let $L=k[x]/(x^3+a_2 x^2+a_1 x+a_0)$ where $f=x^3+a_2 x^2+a_1 x+a_0$ is a separable irreducible polynomial in $k[x]$, then the field extension $L/K$ is Galois if and only if there exists an element $\sigma\in Gal(L/K)$ of order $3$ such that,
\[\sigma:x \mapsto \frac{\alpha x+\beta}{\gamma x+\delta}\]
where $\alpha, \beta, \gamma,\delta\in k$ with $\alpha\delta-\beta\gamma\neq0$.
\end{lemma}
\begin{proof}
As $f$ is a separable irreducible polynomial of degree 3, the extension $L/K$ is separable of degree 3. It is then Galois if and only if there exists $\sigma\in \mathrm{Aut}(L /K)$ of order $3$, so it remains to prove that we can choose $\sigma$ with the right form. If $\sigma\in \mathrm{Aut}(L /K)$ where, $\sigma:x\mapsto \nu_2 x^2+\nu_1 x+\nu_0$ and $\nu_i\in k$ for $i=0,1,2$, so the question here is can we find $\{\alpha, \beta, \gamma,\delta\}\subset K$ with $\alpha\delta-\beta\gamma\neq0$, such that the following equality holds.
\begin{equation}
    \nu_2 x^2+\nu_1 x+\nu_0=\frac{\alpha x+\beta}{\gamma x+\delta}
\end{equation}
  we can find a solution 
\begin{align*}
    \alpha &= a_2 \nu_1 \nu_2 -a_1 \nu_2^2 + \nu_0 \nu_2 - \nu_1^2\\ 
    \beta &= a \nu_0 \nu_2 - a_0 \nu_2^2 - \nu_0\nu_2\\
     \delta &= a_2 \nu_2 - \nu_1\\
     \gamma&=\nu_2.
\end{align*}
We observe that $\alpha\delta-\beta\gamma\neq0$, otherwise we have $\sigma(x)\in K$ and this gives a contradiction as $x\notin K$. \end{proof}
\begin{theorem}\label{thm:cases} Let  $P\in \mathbb{P}^2$, let $C\subset \mathbb{P}^2$ be an irreducible curve of degree $d$: with multiplicity $m_{P}$ at $P$. Thus by Lemma \ref{lemm:deg}, $[k(C):K_{P}]=d-m_{P}$, hence we consider the following cases.
\begin{enumerate}
    \item If $d-m_P=1$, then $\pi_P:C\dashrightarrow \mathbb{P}^1$ is a birational.
    \item If $d-m_P=2$, $P$ is Galois if and only if the extension is separable, and if this holds, then the non-trivial element $\sigma\in G_P$ of order $2$ extends to a de Jonquières map with respect to $P$.
   \item If $d-m_P=3$, and $P$ is Galois, then there is a de Jonquières map with respect to $P$ extending the action. 
\end{enumerate}
\end{theorem}
\begin{proof} We may change coordinates and assume $P=[1:0:0]$. Let $x = X/Z$ and $y = Y/Z$ be affine coordinates. Since the field extension $k(C)/\pi_{P}^{*}(k(\mathbb{P}^1))$ is of degree $d-m_{P}$, then $k(C)=k(y)[x]/(f)$, where $f\in k[x,y]$ is the equation of $C$ in these affine coordinates. 
\begin{itemize}
    \item [(1)] If $d-m_p=1$, then $k(C)=\pi_{P}^{*}(k(\mathbb{P}^1))$ so $\pi_{P}^{*}:k(\mathbb{P}^1)\rightarrow k(C)$ is an isomorphism, so $\pi_{P}:C\dashrightarrow \mathbb{P}^1$ is birational.
    \item [(2)] If $d-m_p=2$, then the extension $k(C)/\pi_{P}^{*}(k(\mathbb{P}^1))$ is of degree $2$ and it is thus Galois if and only if it is separable. $k(C)/\pi_{P}^{*}(k(\mathbb{P}^1))$ is Galois $\Leftrightarrow$ there exists element $\sigma\in G_{P}$ of order $2$ that permutes the roots of $f~\Leftrightarrow~f$ is separable $\Leftrightarrow$ the extension is separable. Moreover, the element $\sigma\in G_P$ of order $2$ is given by $x\mapsto-x$ up to a suitable change of coordinates.
    \item [(3)] If $d-m_{P}=3$, one may write the equation of the curve in the following form $f=F_{d-3}(y,1)x^3+F_{d-2}(y,1)x^2+F_{d-1}(y,1)x+F_{d}(y,1)$, so we get the result directly from Lemma \ref{Lemma:Galois} by replacing $k$ by $k(y)$.\end{itemize}\end{proof}
 \begin{lemma}
      Let $k$ be a field with $char(k)\neq3$ that contains a primitive third root of unity. Let $\phi:\mathbb{P}^1 \rightarrow \mathbb{P}^2$ given by $\phi:[u:v]\mapsto[uv^2+u^2v:u^3: v^3]$. We define $C:=\overline{\phi(\mathbb{P}^1)}$ is a curve of $\mathbb{P}^2$, then the point $P=[1:0:0]$ is a Galois point of $C$ and the extension induced by the projection $\pi_P: C\dashrightarrow \mathbb{P}^1$ is Galois of degree $3$. The element of order $3$ extends to an element of $\mathrm{Jonq}_{P}$
 \end{lemma}
 \begin{proof}
The curve $C$ is birational to $\mathbb{P}^1$ via $\phi$, with inverse $[X:Y:Z]\mapsto [X +Y: X +Z]$. Define the projection by $\pi_P:[X:Y:Z]\mapsto[Y:Z]$. Let $\psi=\pi_{P}\circ\phi$, then $\psi:\mathbb{P}^1\rightarrow\mathbb{P}^1$ maps $[u:v]$ to $[u^3:v^3]$, which is a $3:1$ map, and the extension is Galois of degree $3$ with Galois group $G_{P}$ generated by $\sigma:x\mapsto \omega \cdot x$, where $\omega$ is a primitive cubic root of unity. We have the following commutative diagram
\[\begin{tikzcd}[column sep=large]
    \mathbb{P}^1 \arrow[rightarrow]{r}{\phi}  \arrow[swap]{rd}{\psi} & \mathbb{P}^2 \arrow[dashrightarrow]{d}{\pi_P} \\
        &  \mathbb{P}^1.
    \end{tikzcd}\]
By Theorem \ref{thm:cases}, we know that every element of order $3$ extends to an element of $\mathrm{Jonq}_{P}$. Explicitly $\sigma$ extends to the map that is given by \[[X:Y:Z]\mapsto[\frac{\left(Y-\omega Z \right)X +Y Z(1- \omega)}{(\omega -1)X +Y\omega -Z}:Y:Z].\]
\end{proof}
  \begin{lemma}
     Let $k$ be a field with $char(k)=3$ and $C\subset \mathbb{P}^2$ given by the polynomial $f= X^{3}- Y^{2} X+Z^{3}$, then the point $P=[1:0:0]$ is Galois point of $C$ and the extension induced by the projection $\pi_P: C\dashrightarrow \mathbb{P}^1$ is Galois of degree $3$.
 \end{lemma}
\begin{proof}
Define the birational map $\phi:\mathbb{P}^1 \dashrightarrow C$ by $ \phi:[u:v] \mapsto [v^{3}: u^{3} : u^{2}  v- v^3]$ with inverse $[X:Y:Z]\mapsto[Y:Z+X]$. Define the projection by $\pi_P:[X:Y:Z]\mapsto[Y:Z]$. Let $\psi=\pi_{P}\circ\phi$, then $\psi:\mathbb{P}^1\rightarrow\mathbb{P}^1$ maps $[u:v]$ to $[u^{3} : u^{2}  v- v^3]$, which is a $3:1$ map, and the extension is Galois of degree $3$ with Galois group $G_{P}$ generated by $\sigma:[u:v]\mapsto[u:u+v]$. We have the following diagram
\[\begin{tikzcd}[column sep=large]
 \mathbb{P}^1 \arrow[rightarrow]{r}{\phi}  \arrow[swap]{rd}{\psi} & \mathbb{P}^2 \arrow[dashrightarrow]{d}{\pi_P} \\
       &  \mathbb{P}^1.
    \end{tikzcd}\]
By Theorem \ref{thm:cases}, we know that every element of order $3$ extends to an element of $\mathrm{Jonq}_{P}$. Explicitly $\sigma$ extends to the map that is given by $\sigma$ that is given by $[X:Y:Z]\mapsto[X +Y :Y:Z]$.
 \end{proof}
\section{Curves that are Cremona equivalent to a line}
\begin{definition}
Let $X$ be a smooth projective variety and $D$ a divisor in $X$. Let $K_X$ denote a canonical divisor of $X$. We define the Kodaira dimension of $D\subset X$, written $\mathcal{K}(D,X)$ to be the dimension of the image of $X \mapsto P(H^{0}(m(D+ K_X)))$ for $m>>0$. By convention we say that the Kodaira dimension is $\mathcal{K}(D,X)= -\infty$ if $|m(D+ K_X)|=\phi ~\forall m>0$.
\end{definition}
 \begin{theorem}\cite[Corollary $2.4$]{MR685529}\label{cor:kod}
 Let $X$ be smooth rational surface and $D\subset X$ is a curve isomorphic to $\mathbb{P}^1$, then the following are equivalent:
 \begin{enumerate}
   \item $\mathcal{K}(D,X)=-\infty$
  \item $|2K_{X}+D|=\phi.$
  \item $|2(K_{X}+D)|=\phi.$
   \end{enumerate}
\end{theorem}
 \begin{definition}
     Let $C\subset \mathbb{P}^2$ be an irreducible curve. If $C$ is a smooth curve then we define $\bar{\mathcal{K}}(C,\mathbb{P}^2)$ to be $\mathcal{K}(C,\mathbb{P}^2)$, and if $C$ is a singular curve, we take $X\rightarrow \mathbb{P}^2$ to be an embedded resolution of singularities of $C$ in $\mathbb{P}^2$ where $\Tilde{C}$ is the strict transform of $C$, then we define $\bar{\mathcal{K}}(C,\mathbb{P}^2)$ to be $\mathcal{K}(\Tilde{C},X)$. By \cite{MR685529}, this does not depend on the choice of the resolution.
 \end{definition}
\begin{definition}
Let $C$ be an irreducible smooth plane curve. The curve $C$ is said to be Cremona equivalent to a line if there is a birational map $\varphi:\mathbb{P}^2\dashrightarrow\mathbb{P}^2$ that sends $C$ to a line.
\end{definition}
\begin{theorem}(Coolidge) \label{thm_coolidge}\cite[Theorem $2.6$]{MR685529} 
Let $C\hookrightarrow \mathbb{P}^2 $ be an irreducible rational curve. Then there exists a Cremona transformation $\sigma$ of $\mathbb{P}^2$ such that $\sigma(C)$ is a line if and only if $\overline{\mathcal{K}}(C,\mathbb{P}^2)=-\infty$.
\end{theorem}
\begin{lemma}
    If $C\hookrightarrow \mathbb{P}^2 $ is an irreducible rational curve of degree $d<6$, then $C$ is equivalent to a line.
\end{lemma}
\begin{proof}
Let $\pi_1:X_1 \rightarrow \mathbb{P}^2$ be the blow-up of $\mathbb{P}^2$ at $P_1$, and let $\pi_i:X_{i} \rightarrow X_{i-1}$ the blow-up of $X_{i-1}$ at $P_i\in X_{i-1}$ for $i\geq2$, $\mathcal{E}_{i}=\pi_i^{-1}(P_i)$ is a $(-1)$-curve, where $\mathcal{E}_{i}^2=-1$ and $\mathcal{E}_{i}\cong \mathbb{P}^1$. After blowing up $n$ points, let $\pi: Y\mapsto\mathbb{P}^2 $ be the composition of the blow-ups $\pi_i$, where $Y=X_n$ we choose enough points such that the strict transform of $C$ is smooth. By induction, we have $E_i=(\pi_{i+1}\circ.....\circ \pi_{n})^{*}(\mathcal{E}_i)$, $Pic(Y)=\pi^*(Pic(\mathbb{P}^2))\oplus \mathbb{Z}E_{1}\oplus...\oplus \mathbb{Z}E_{n}$, and $E_i^2=-1$ for every $i=1,...,n,$ and $E_i\cdot E_j=0$ for every $i\neq j$. Moreover,
\begin{align*}
K_{Y}&=\pi_{n}^*.... \pi_1^*(K_{\mathbb{P}^2})+\Sigma_{i=1}^{n}\pi_{n}^*.... \pi_{i+1}^*(\epsilon_{i})\\
&=\pi^*(K_{\mathbb{P}^2})+\Sigma_{i=1}^{n}E_{i}\\ 
&=-3\pi^*(L)+\Sigma_{i=1}^{n}E_{i}.
\end{align*}
The strict transform $\Tilde{C}\subset C$ is equivalent to $\Tilde{C}=d \cdot \pi^*(L)-\Sigma_{i=1}^{n}m_{P_i}(C)E_{i}$. Hence we have
\begin{align*}
2 K_Y+\Tilde{C}=(-6+d)\cdot \pi^*(L)+\Sigma_{i=1}^{n}(2-m_{P_i}(C))E_{i},
\end{align*}
so $\pi^*(L)\cdot(2 K_{Y}+\Tilde{C})=-6+d$, hence $|2 K_{Y}+\Tilde{C}|=\phi$ for every curve of degree $d<6$ and according to \ref{cor:kod}, we have $\bar{\mathcal{K}}(C,\mathbb{P}^2)=-\infty$, so $C$ is equivalent to a line.
\end{proof}
\begin{lemma}
    If $C$ is a Cremona equivalent to a line $L\subseteq \mathbb{P}^2$, $P$ is a Galois point, then every non-trivial element in $G_{P}$ extends to an element in $\mathrm{Bir}(\mathbb{P}^2)$.
\end{lemma}
\begin{proof}
    Let $\varphi\in \mathrm{Bir}(\mathbb{P}^2)$ that sends $C$ onto a line $L$. For each $g\in G_{P},~\varphi|_{C}:C\dashrightarrow L$ conjugates $g$ to an element of $\mathrm{Aut}(L)$, that extends to $\hat{g}\in \mathrm{Aut}t(\mathbb{P}^2)$. Hence, $g$ extends to $\varphi^{-1}\hat{g}\varphi\in \mathrm{Bir}(\mathbb{P}^2)$.
\end{proof}
\begin{remark}\label{Rem:conic}
     Let $C$ be the smooth conic given by $C= \{[X : Y : Z] | Y^2=XZ\}\subset \mathbb{P}^2$, then the natural embedding of $\mathrm{Aut}(\mathbb{P}^2, C)=\{g\in \mathrm{Aut}(\mathbb{P}^2)|g(C)=C\}= PGL_2$ in $\mathrm{Aut}(\mathbb{P}^2) = PGL_3$ is the one induced from the injective group homomorphism
     \begin{equation*}
         GL_2(k)\rightarrow GL_3(k),~ \left[\begin{array}{cc}
             a & b \\
             c & d
         \end{array}\right]\mapsto \frac{1}{ad-bc}\left[\begin{array}{ccc}
             a^2 & ab&b^2 \\
             2ac & ad+bc&2bd\\
             c^2&cd&d^2
         \end{array}\right]
     \end{equation*}
     where $\rho:[u:v]\mapsto[u^2:uv:v^2]$, and the following diagram commutes.
\[\begin{tikzcd}[column sep=small]
\mathbb{P}^1\arrow[rr, dashed,->,"\rho"] \arrow[d,"g"]&&\tilde{C}\subset  \mathbb{P}^2 \arrow[d,"\tilde{g}"] \\
\mathbb{P}^1\arrow[rr, "\rho"] & & \mathbb{P}^2 
\end{tikzcd}\]
\end{remark}
 \begin{lemma}\label{toBir}
 Let $k$ be a field of characteristic $char(k)\neq2$ containing a primitive fourth root of unity, and let $C$ be the irreducible curve defined by the equation 
\begin{equation}
X^{4}-4 Z Y X^{2}-Z Y^{3}+2 Z^{2} Y^{2}-Y Z^{3} =0,
\end{equation}
then the point $P=[1:0:0]$ is an outer Galois point of $C$ and the extension induced by the projection $\pi_P: C\dashrightarrow \mathbb{P}^1$  is Galois of degree $4$. The group $G_{P}$ extends to $\mathrm{Bir}(\mathbb{P}^2)$ but not to $\mathrm{Jonq}_{P}$.
\end{lemma}
\begin{proof}
Define the birational map $\phi:\mathbb{A}^1 \dashrightarrow C$ by $\phi:t\mapsto[t+t^3:t^4:1]$ with inverse $[X:Y:Z]\mapsto \left(X \left(Y +Z \right)\right)/(X^{2}-Y Z +Z^{2})$. Hence $C$ is a rational irreducible curve of degree $4$. Therefore, because of Coolidge theorem \ref{thm_coolidge}, we know that every non-trivial element in $G_{P}$ extends to an element in $\mathrm{Bir}(\mathbb{P}^2)$. We will also prove it explicitly below.  We have $K(C)=k(t)$ and define the projection by $\pi_P:[X:Y:Z]\mapsto[Y:Z]$. Let $x = X/Z$ and $y = Y/Z$ be affine coordinates. Hence the affine equation $x^{4}-4 y x^{2}- y^{3}+2 y^{2}-y  =0$ is defining the extension field $k(y)[x]/k(y)=k(t)/k(t^4)$. As $k$ contains the $4th$ root of unity, the extension is Galois of degree $4$ with basis $\{1,t, t^2,t^3\}$, and we have the following diagram
\[\begin{tikzcd}[column sep=large]
    \mathbb{A}^1 \arrow[dashrightarrow]{r}{\phi}  \arrow[swap]{rd}{\psi} & C \subseteq  \mathbb{P}^2 \arrow{d}{f} \\
        &  \mathbb{A}^1
    \end{tikzcd}\]
 where $\psi$ is a $4:1$ map given by $\psi:t\mapsto t^4$. By contradiction, we prove that there is a de Jonquières map $f$ extending the action. Let us assume that there exists a de Jonquières map $g$ that extends the action, $i.e$ there exists $\tilde{\alpha},\tilde{\beta}, \tilde{\gamma},\tilde{\delta}\in k(y)$ with $\tilde{\alpha}\tilde{ \delta}-\tilde{\beta} \tilde{\gamma}\neq0$ such that $g:(x,y)\mapsto(\frac{\tilde{\alpha} x+\tilde{\beta}}{\tilde{\gamma} x+\tilde{\delta}},y)$. As $g\circ\phi=\phi\circ\sigma$, writing $\alpha=\tilde{\alpha}(t^4),\beta=\tilde{\beta}(t^4), \gamma=\tilde{\gamma}(t^4)$ and $\delta=\tilde{\delta}(t^4)$ where $\alpha,\beta,\gamma, \delta\in k(t^4)$. We obtain the equation
\[\mathrm{i}t-\mathrm{i}t^3=\frac{\alpha (t+t^3)+\beta}{\gamma (t+t^3)+\delta}.\] This gives $\beta = \beta(t)=-(\mathrm{i}   t^{6}- \mathrm{i})\gamma t^{2}-(\mathrm{i} \delta +\alpha ) t^{3}+(\delta\mathrm{i}-\alpha )t\in k(t^4)$ and then we have $(\mathrm{i} t^{4}- \mathrm{i})\gamma=0, \mathrm{i} \delta +\alpha=0$ and $\delta\mathrm{i}-\alpha=0$. This gives $\alpha=0,\gamma=0$ and this gives a contradiction, so there is no de Jonquières map $f$ extending the action.\\
Viewing $C$ as an irreducible curve in $\mathbb{P}^2$ of degree $4$, there are three singular points on the curve $ [0: 1: 1], \left[\mathrm{i} \sqrt{2}:-1: 1\right], \left[\mathrm{i} \sqrt{2}: 1: -1\right]$. After suitable change of coordinates $\sigma:\mathbb{P}^2\rightarrow\mathbb{P}^2$ given by $[X:Y:Z]\mapsto \left[-\mathrm{i} \sqrt{2} X-\mathrm{i} \sqrt{2} Z: 2 X+Y-Z : 2X-Y+Z \right]$, this map sends the curve $C$ to $\tilde{C}$, which is given by $\tilde{f}=X^{2} Y^{2}+6 X^{2} Y Z +X^{2} Z^{2}+4 Y^{2} Z^{2}=0$ and this new equation has $\{[1:0:0], [0:1:0],[0:0:1]\}$ as multiple points of order $2$. After blowing up the three points $\{[1:0:0], [0:1:0],[0:0:1]\}$ in $\mathbb{P}^2$ and contract again, we find that the strict transform curve is of degree $d'=2 \cdot d-m_1-m_2-m_3=2\cdot4-2-2-2=2$, so it is a conic given by the equation $F=4 X^{2}+Y^{2}+6 Y Z +Z^{2}=0$, after a suitable change of coordinates using the following matrix,
 \[\left[\begin{array}{ccc}
4\mathrm{I} & 0 & -\mathrm{I} 
\\
 0 & 2\sqrt{2} & 0 
\\
8 & -6 \sqrt{2} & 2
\end{array}\right]\]
we can send the conic to $Y^2-XZ=0$. We can then extend $G_{P}$ explicitly using Remark \ref{Rem:conic}
\end{proof}
\section{Example where $G_{P}$ can not be extended to $Bir(\mathbb{P}^2)$}
\begin{lemma}\label{multip}
 Let k be an algebraically closed field, $C\subset\mathbb{P}^2$ be an irreducible curve, $f: \mathbb{P}^2\dashrightarrow \mathbb{P}^2$ be a birational map sends the curve $C$ to itself, and $X\rightarrow \mathbb{P}^2$ is an embedded resolution of singularities of $C$ in $\mathbb{P}^2$ where $\Tilde{C}$ is the strict transform of $C$. If all singular points of $C$ have a multiplicity $m_P(C)<deg(C)/3$, then $f$ is an automorphism of $\mathbb{P}^2$.
\end{lemma}
\begin{proof}
Let $deg(f)=d$, and assume for contradiction that $d>1$. We take a commutative diagram where $\pi$ and $\eta$ are sequences of blow-ups
\[\begin{tikzcd}
     & \tau & {} \arrow[lld, shift right] X \arrow[rrd, shift left] & \eta &               \\
 \mathbb{P}^2 \arrow[rrrr, dashed, shift right,"f"] &  &&  &  \mathbb{P}^2
\end{tikzcd}\]
and we can assume that the strict transform of $C$ is smooth, then we have 
  \begin{align*}
     \eta^{*}(L)&=d\cdot\pi^{*}(L)-\Sigma m_iE_i \\
    K_{X}&=-3\pi^{*}(L)+\Sigma E_i\\
    \eta^{*}(C)&=\Tilde{C}=deg(C)\cdot\pi^{*}(L)-\Sigma m_{P_{i}}(C)E_i.
\end{align*}
But $deg(C)=C\cdot L=\eta^{*}(C)\cdot\eta^{*}(L)=d\cdot deg(C)-\Sigma m_i\cdot m_{P_{i}}(C)$, hence $deg(C)(d-1)=\Sigma m_i\cdot m_{P_{i}}(C)<\Sigma m_i\cdot deg(C)/3$, and then $3(d-1)<\Sigma m_i$ is a contradiction as $\Sigma m_i=3(d-1)$.
\end{proof}
\begin{lemma}\label{not-to-Bir}
Let $k$ be a field with $char(k)\neq5$ that contains a primitive $5$th root of unity, and let $\phi:\mathbb{P}^1 \rightarrow \mathbb{P}^2$ given by \[\phi:[u:v]\mapsto[u v^6-u^7:u^5(u^2+v^2): v^5(u^2+v^2)].\] We define $C:=\overline{\phi(\mathbb{P}^1)}$ is a curve of $\mathbb{P}^2$, then the point $P=[1:0:0]$ is an inner Galois point of $C$ and the extension induced by the projection $\pi_P: C\dashrightarrow \mathbb{P}^1$ is Galois of degree $5$. Moreover, there is no birational map $f$ extending the action of the generator of the Galois group. 
\end{lemma}
\begin{proof}
The curve $C$ is birational to $\mathbb{P}^1$ via $\phi$, with inverse $[X:Y:Z]\mapsto [X^{4} Y +4 X^{3} Y^{2}-2 X^{3} Z^{2}+6 X^{2} Y^{3}-2 X^{2} Y \,Z^{2}+4 X \,Y^{4}+X \,Z^{4}+Y^{5}:Z \left(X^{4}+2 X^{3} Y +X^{2} Y^{2}-3 X^{2} Z^{2}-3 X Y \,Z^{2}+Z^{4}\right)
]$. Define the projection by $\pi_P:[X:Y:Z]\mapsto[Y:Z]$. So let $\psi=\pi_P\circ\phi$ then $\psi:\mathbb{P}^1\rightarrow\mathbb{P}^1$ maps $[u:v]$ to $[u^5:v^5]$, which is a $5:1$ map, and the extension is Galois of degree $5$ with Galois group $G_{P}$ generated by $\sigma:x\mapsto \zeta\cdot x$, where $\zeta$ is the $5$th root of unity and we have the following diagram
\[\begin{tikzcd}[column sep=large]
    \mathbb{P}^1 \arrow[rightarrow]{r}{\phi}  \arrow[swap]{rd}{\psi} & \mathbb{P}^2 \arrow[dashrightarrow]{d}{\pi_P}\\
        &  \mathbb{P}^1.
    \end{tikzcd}\]
We now prove that the curve $C$ does not have a point $P$ of multiplicity $m_P(C)\geq 3$. By contradiction, we take a point $P=[P_0:P_1:P_2]$ of multiplicity $3$, and then we take two distinct lines $a_{1}x+a_{2}y-a_{3}z=0$ and $b_{2}y-b_{3}z=0$ passing through the point $P$. We take the preimage in $\mathbb{P}^1$, so we get a common factor of degree at least $3$. 
\begin{align*}
   f_{1}(u,v) &= a_{1} \left(-u^{7}+u \,v^{6}\right)+a_{2} u^{5} \left(u^{2}+v^{2}\right)+a_{3} v^{5} \left(u^{2}+v^{2}\right),\\
  f_{2}(u,v) &=\left(u^{2}+v^{2}\right) \left(b_{2} u^{5}+b_{3} v^{5}\right).
\end{align*}
We check now that it is not possible for the polynomials $f_1$ and $f_2$ to have a factor of degree $3$ in common. Assume first that $u +\mathrm{I} v$ divides both polynomials, so we should have $f_{1}(1,\mathrm{I})=f_{2}(1,\mathrm{I})=0$ implies to $a_{1}=0$, hence $P=[1:0:0]$ is a smooth point and this gives a contradiction. If we assume that $u -\mathrm{I} v$ divides both polynomials, so we should have $f_{1}(1,-\mathrm{I})=f_{2}(1,-\mathrm{I})=0$, again we have $a_{1}=0$, so the factor of degree $3$ must divide $b_{2} u^{5}+b_{3} v^{5}$. If we assume that $u$ divides the polynomial $f_{2}$, then $b_{3}=0$ and $u^3$ should divide $f_{1}$, but this is not true as $f_{1}=-u(u^6 -  v^6)$. If we assume that $v$ divides the polynomial $f_{2}$, then $b_{2}=0$ and $v^3$ should divide $f_{1}$, but this is not true as $f_{1}=uv^2(u^4 + v^4)$. So the factor of degree $3$ must divide $b_{2} u^{5}+b_{3} v^{5}$, $b_{2}\neq0$ and $b_{3}\neq0$. Hence we can assume that $a_{1}=1$ and replace $f_{1}$ by $f_{1}-a_{3}f_{2}/b_{3}$ so we can put $a_{3}=0$ and up to multiple, we can assume that $b_{3}=1, a_{1}=1$ and $b_{2}=-\xi^{5}$. So $f_{2}=-u^{5} \xi^{5}+v^{5}$, 
$f_{1}=\left(-u^{7}+u \,v^{6}\right)+a_{2} u^{5} \left(u^{2}+v^{2}\right)$ 
so the roots of $f_{2}$ are $(u,v)=(1,\xi\zeta^i)$, where $\zeta$ is a $5-$th root of unity. and since $f_{1}$ and $f_{2}$ should have three roots in common, therefore let $\{(1,\xi),(1,\xi\zeta),(1,\xi\rho)\}$ are the three roots in common where $\rho^5=1$ and $\zeta^5=1$, and  $\rho^5\neq\zeta$ and they are not equal to $1$, so $f_{1}$ should vanish on these three roots. This gives three equations 
\begin{align*}
    q_{1}&=\xi^{6}-1+a_{2} \left(\xi^{2}+1\right)=0\\
    q_{2}&=a_{2} \left(\xi^{2} \zeta^{2}+1\right)+\xi^{6} \zeta -1=0\\
    q_{3}&=a_{2} \left(\rho^{2} \xi^{2}+1\right)+\rho  \,\xi^{6}-1=0,
\end{align*}
 by solving this system in $a_{2}, \zeta$ and $\rho$, we found that $\zeta=\rho=(\xi^{4}+1)/(\xi^{6}-1)$, which is a contradiction as $\zeta\neq\rho$. So $f_{1}$ and $f_{2}$ can not have a factor of degree $d\geq3$. Since $m_P(C)<3$ for each $P\in \mathbb{P}^2$, let us assume that there exists a birational map $g$ that extends the generator of the Galois group, then by Lemma \ref{multip}, $g$ is a linear automorphism of $\mathbb{P}^2$, so it is given by a matrix lets say $A\in PGL_{3}$, Since $g\circ\phi=\phi\circ\sigma$, so we have
 \[\left[\begin{array}{c}
      U V^6-U^7  \\
      U^5(U^2+V^2)\\
      V^5(U^2+V^2)
 \end{array}\right]=A\cdot \left[\begin{array}{c}
     \zeta U V^6-\zeta^2 U^7   \\
      U^5(\zeta^2 U^2+V^2)\\
      V^5(\zeta^2 U^2+V^2)
 \end{array}\right]\]
 where $\zeta$ is the $5th$ root of unity. Since $ U V^6, U^7, U^5 V^2,$ and $V^7$ are linearly independent, after checking the calculation we found that $A$ should be diagonal, but $U^5(\zeta^2 U^2+V^2)$ is not a multiple of $U^5(\zeta^2 U^2+V^2)$, then we have a contradiction.  
\end{proof}
\bibliographystyle{alphadin}
\bibliography{reference}

\begin{thebibliography}{MKM83}


\providecommand{\url}[1]{\texttt{#1}}
\expandafter\ifx\csname urlstyle\endcsname\relax
  \providecommand{\doi}[1]{doi: #1}\else
  \providecommand{\doi}{doi: \begingroup \urlstyle{rm}\Url}\fi

\bibitem[Fuk09]{Fukasawa2009galois}
\textsc{Fukasawa}, Satoru:
\newblock Galois points for a plane curve in arbitrary characteristic.
\newblock {In: }\emph{Geom. Dedicata} 139 (2009), 211--218.
\newblock \url{https://doi.org/10.1007/s10711-008-9325-2}. --
\newblock ISSN 0046--5755

\bibitem[Har77]{MR0463157}
\textsc{Hartshorne}, Robin:
\newblock \emph{Algebraic geometry}.
\newblock Springer-Verlag, New York-Heidelberg, 1977 (Graduate Texts in
  Mathematics, No. 52). --
\newblock  xvi+496 S. --
\newblock ISBN 0--387--90244--9

\bibitem[Miu08]{MR2427627}
\textsc{Miura}, Kei:
\newblock Galois points for plane curves and {C}remona transformations.
\newblock {In: }\emph{J. Algebra} 320 (2008), Nr. 3, 987--995.
\newblock \url{https://doi.org/10.1016/j.jalgebra.2008.04.018}. --
\newblock ISSN 0021--8693

\bibitem[Miu18]{Mr1048550}
\textsc{Miura}, Kei:
\newblock Birational transformations belonging to Galois points for a certain
  plane quartic.
\newblock   (2018).
\newblock \url{https://arxiv.org/abs/1708.08195}

\bibitem[MKM83]{MR685529}
\textsc{Mohan~Kumar}, N. ; \textsc{Murthy}, M.~P.:
\newblock Curves with negative self-intersection on rational surfaces.
\newblock {In: }\emph{J. Math. Kyoto Univ.} 22 (1982/83), Nr. 4, 767--777.
\newblock \url{https://doi.org/10.1215/kjm/1250521679}. --
\newblock ISSN 0023--608X

\bibitem[MY00]{miura2000field}
\textsc{Miura}, Kei ; \textsc{Yoshihara}, Hisao:
\newblock Field Theory for Function Fields of Plane Quartic Curves.
\newblock {In: }\emph{Journal of Algebra} 226 (2000), Nr. 1, 283-294.
\newblock
  \url{https://www.sciencedirect.com/science/article/pii/S0021869399981735}. --
\newblock ISSN 0021--8693

\bibitem[Yos09]{Yoshihara2009rational}
\textsc{Yoshihara}, Hisao:
\newblock Rational curve with {G}alois point and extendable {G}alois
  automorphism.
\newblock {In: }\emph{J. Algebra} 321 (2009), Nr. 5, 1463--1472.
\newblock \url{https://doi.org/10.1016/j.jalgebra.2008.11.035}. --
\newblock ISSN 0021--8693

\end{thebibliography}
\end{document}